\documentclass{article}
\usepackage[utf8]{inputenc}
\usepackage[english]{babel}
\usepackage{authblk}

\usepackage{amsmath, amsthm, amssymb}
\usepackage{url}

\newtheorem{theorem}{Theorem}[section]

\newtheorem{lemma}[theorem]{Lemma}

\newtheorem{proposition}[theorem]{Proposition}
\newtheorem{corollary}[theorem]{Corollary}
\theoremstyle{remark}

\newtheorem{remark}[theorem]{Remark}

\title{A class of fields with a restricted model completeness property}

\author{Philip Dittmann\thanks{philip.dittmann@tu-dresden.de}}
\affil{Institut für Algebra, TU Dresden}
\author{Dion Leijnse\thanks{d.leijnse@gmail.com}}
\affil{Radboud Universiteit Nijmegen}
\date{8 November 2019}

\begin{document}

\maketitle


\begin{abstract}
    We introduce and study a natural class of fields in which certain first-order definable sets are existentially definable, and characterise this class by a number of equivalent conditions.
    We show that global fields belong to this class, and in particular obtain a number of new existential (or diophantine) predicates over global fields.
\end{abstract}

\section{Introduction}

Model completeness is recognised as an important property of certain classes of well-behaved fields by model theorists and algebraic geometers.
Model completness can be defined as the property that every first-order definable set is existentially definable.
This is related to
statements in geometry about projections of varieties, such Chevalley's Theorem for algebraically closed fields (projections of constructible sets are constructible) or the Tarski--Seidenberg Theorem for real closed fields.

In practice, however, model completeness is not as common as one might wish. 
In particular, in research surrounding Hilbert's Tenth Problem one is frequently concerned with finding existential (or \emph{diophantine}) definitions for properties of interest in fields which are not model complete, such as global fields.
In this article, we introduce and study a new class of fields which includes all model complete fields, or more precisely includes all fields with model complete elementary diagram, and notably also includes all global fields, i.e.\ finite extensions of $\mathbb{Q}$ or $\mathbb{F}_p(T)$.

Our main theorem is the following characterisation result.
\begin{theorem}
  The following are equivalent for a field $K$.
  \begin{enumerate}
    \item[(1)] For every pair of elementary extensions $K^{**}, K^*$ of $K$ with $K^{**}\supseteq K^*$, the extension $K^{**}/K^*$ is a regular extension of fields.
    \item[(2)] For every 3-tuple $(m,d,r)$ there is an existential criterion for the ideal $(f_1, \dots, f_r) \subseteq K[X_1, \dots, X_m]$ generated by polynomials $f_i$ of total degree at most $d$ to be prime.
    \item[(3)] For every 3-tuple $(m,d,r)$ there is an existential criterion for the ideal $(f_1, \dots, f_r) \subseteq K[X_1, \dots, X_m]$ generated by polynomials $f_i$ of total degree at most $d$ to be maximal.
  
  \end{enumerate}
  These statements imply the following, and are equivalent to it if $K$ has finite degree of imperfection, i.e.\ if the extension $K^{1/p}/K$ is finite, where $p$ is the characteristic exponent of $K$.
  \begin{enumerate}
    \item[(4)] For every quantifier-free formula $\psi(x, \overline{y})$ in the language of rings where $x$ is a distinguished variable, the formula $\forall x \psi(x, \overline{y})$ is equivalent over $K$ to an existential formula with parameters in $K$.
  \end{enumerate}
\end{theorem}
Here by an existential criterion we mean an existential formula in the language of rings with parameters in $K$, where the free variables of the formula correspond to the coefficients of the polynomials $f_i$.

As results from \cite{Dittmann18} imply that global fields satisfy (1), the main theorem in particular implies that global fields satisfy (2), (3) and (4), so we gain new existential definitions in global fields (Corollary \ref{cor:global_fields}).

\paragraph{Acknowledgements} This article grew out of a research visit of the second author to the first author as part of the bachelor's programme of the Radboud Honours Academy.
The second author would like to thank Arne Smeets for being his thesis supervisor, and gratefully acknowledges travel support by the Radboud Honours Academy.
We would both like to thank to Arno Fehm for helpful comments on a preliminary version of this article.

\section{Proof of the main theorem}

We begin by collecting some lemmas regarding first-order definability of prime and maximal ideals in polynomial rings.
When speaking of first-order properties of a collection of polynomials $f_1, \dotsc, f_r$ of total degree at most $d$ in a polynomial ring $K[X_1, \dotsc, X_m]$ over a field, we fix some ordering of the monomials and associate to the $f_i$ their tuple of coefficients, so that a first-order property of the $f_i$ will just mean a first-order property of their tuple of coefficients.

\begin{lemma}[{\cite[Statements (I) and (II)]{Dries84}}]
\label{IdealBounds}
  Fix natural numbers $m, d, r$ and let $I = (f_1, \dots, f_r)$ be an ideal of $K[X_1, \dots, X_m]$, where all $f_i$ have total degree at most $d$.  Then:
  \begin{enumerate}
      \item[(1)] Let $f \in K[X_1, \dots, X_m]$ be of total degree at most $D$. There is a constant $A$ depending only on $m, d$ and $D$ (in particular not on $K$) such that $f \in I$ if and only if there exist $h_i \in K[X_1, \dots, X_m]$ of total degree at most $A$ with $f = h_1f_1 + \dots + h_rf_r$.
      \item[(2)] There is a constant $B$ depending only on $m$ and $d$ such that the ideal $I$ is prime if and only if $1 \notin I$ and the condition that $fg \in I$ implies $f \in I$ or $g \in I$ is satisfied for all $f,g \in K[X_1, \dots, X_m]$ of total degree at most $B$.
  \end{enumerate}
\end{lemma}

These two results imply that we can write down a first-order formula for checking whether an ideal is prime.

\begin{lemma}
\label{primality-first-order}
  Given a 3-tuple $(m,d,r)$ there is a first-order formula $P_{m,d,r}(\overline{x})$ that is satisfied by a tuple $\overline{a}$ from a field $K$ if and only if the ideal $(f_1, \dots, f_r) \subseteq K[X_1,\dots, X_m]$ determined by the coefficients $\overline{a}$ is prime.
\end{lemma}
\begin{proof}
  By the first part of the above lemma, we can check using a first-order formula that $1 \notin (f_1, \dots, f_r)$, since we need to check that for all $h_i$ of some bounded degree we have $h_1f_1 + \dots + h_rf_r \neq 1$. Using the second part of the lemma it suffices to quantify over all polynomials $f,g$ of degree at most $B$ in order to check that $I$ is prime. For checking the ideal membership we again use the first part of the lemma.
\end{proof}

There also exists a first-order formula for maximal ideals.
\begin{lemma}[{\cite[1.6]{Dries78}}\footnote{We would like to thank Matthias Aschenbrenner for pointing out this reference to us.}]
\label{maximality-first-order}
Given a 3-tuple $(m,d,r)$ there is a first-order formula $M_{m,d,r}(\overline{x})$ that is satisfied by a tuple $\overline{a}$ in a field $K$ if and only if the ideal $(f_1, \dots, f_r) \subseteq K[X_1,\dots, X_m]$ determined by the coefficients $\overline{a}$ is maximal.
\end{lemma}

By combining this fact with the compactness theorem, we can strengthen Zariski's Lemma. While Zariski's Lemma says that the quotient of a polynomial ring by a maximal ideal is a finite field extension, it is often stated and proved without any bounds.
\begin{lemma}
\label{StrongZariski}
Given a 3-tuple $(m,d,r)$ there is a natural number $N$ such that for every maximal ideal $I = (f_1, \dots, f_r)$ generated by polynomials of total degree at most $d$ the dimension $\dim_K K[X_1,\dots, X_m]/I$ is bounded by $N$.
\end{lemma}
The argument for this lemma is implicitly given in \cite[top of page 148]{Dries78}, but we give the full proof.
\begin{proof}
Let $n$ be a natural number. In order to check that the dimension of the quotient $K[X_1, \dots, X_m]/I$ is at most $n$, it suffices to check that every set of $n+1$ monomials of total degree at most $n$ is linearly dependent modulo $I$. There are only finitely many monomials of total degree at most $n$, and the degree of any linear combination of such monomials is bounded by $n$. Using the first part of Lemma \ref{IdealBounds} we can now write down a first-order formula $\varphi_n(\overline{x})$ expressing that $K[X_1, \dots, X_m]/I$ has dimension at most $n$ over $K$. 

Hence the set of formulae $\Phi = \{\neg \varphi_n : n \in \mathbb{N}\} \cup \{M_{m,d,r}\}$ is inconsistent by Zariski's Lemma, so by the compactness theorem there is a finite inconsistent subset $\Phi'\subseteq \Phi$. Let $N$ be the largest number such that $\neg \varphi_N \in \Phi'$, then we see that $\dim_K K[X_1,\dots, X_m]/I \leq N$. 
\end{proof}

In our proof of the main theorem we will repeatedly use the following version of the \L o\'s--Tarski preservation theorem.
\begin{theorem}
  \label{LosTarski}
  Let $K$ be a field and $\varphi(\overline{x})$ a formula in the language of rings with $n$ free variables. Then $\varphi$ is equivalent over $K$ to an existential formula with parameters if and only if for every pair of elementary extensions $K^{**}, K^*$ of $K$ with $K^* \subseteq K^{**}$ and $\overline{a} \in (K^*)^n$ the statement $K^* \models \varphi(\overline{a})$ implies $K^{**} \models \varphi(\overline{a})$ (that is, $\varphi$ is preserved under inclusions of elementary extensions).
\end{theorem}
\begin{proof}
  This is \cite[Corollary 5.4.5]{Hodges97} applied to the diagram of $K$, i.e.\ to the full elementary theory of $K$ in the language of rings expanded by a constant for each field element.
\end{proof}

Recall that an extension of fields $E/F$ is \emph{regular} if $E$ is linearly disjoint from the algebraic closure $\overline{F}$ over $F$, or equivalently if $F$ is relatively algebraically closed in $E$ and $E/F$ is separable (\cite[V.17, Proposition 9]{BourbakiAlg5}).
\begin{proof}[Proof of main theorem (equivalence of first three statements)]
We prove this by showing that (1) implies (2), (2) implies (3) and (3) implies (1).

First suppose that every pair of elementary extensions is regular. By Lemma \ref{primality-first-order} there is a first-order formula for checking that an ideal $(f_1,\dots,f_r)$ is prime.

Let $K^{**},K^*$ be two elementary extensions of $K$ with $K^{**}\supseteq K^*$. Let $I = (f_1, \dots, f_r) \subseteq K^*[X_1, \dots, X_m]$ be a prime ideal, then $K^*[X_1, \dots, X_n]/I$ is an integral domain. There is an isomorphism $$K^{**}[X_1,\dots, X_n]/(f_1,\dots, f_r) \cong K^{**} \otimes_{K^*} K^*[X_1, \dots, X_n]/I.$$ Let $L$ be the quotient field of $K^*[X_1,\dots, X_n]/I$, then $K^{**}\otimes_{K^*}L$ is an integral domain because $K^{**}/K^*$ is regular \cite[V.17, Proposition 9]{BourbakiAlg5}. There is an injection 
$$K^{**}\otimes_{K^*} K^*[X_1,\dots,X_n]/I \to K^{**} \otimes_{K^*}L,$$
so the ideal $(f_1,\dots, f_r)\subseteq K^{**}[X_1,\dots,X_n]$ is also prime. Therefore the property of being a prime ideal is preserved under inclusions of elementary extensions. The \L o\'s--Tarski theorem (Theorem \ref{LosTarski}) now implies that the given first-order formula for being a prime ideal is equivalent to an existential formula over $K$.

Next, assume that there is an existential definition for prime ideals, and fix a 3-tuple $(m,d,r)$. By Lemma \ref{StrongZariski} there is an $N$ such that every maximal ideal $I=(f_1, \dots, f_r)$ for this 3-tuple has $\dim_K K[X_1,\dots, X_m]/I \leq N$. Using Zariski's Lemma, we see that $I$ is maximal if and only if it is prime and $\dim_K K[X_1, \dots, X_m]/I\leq N$. It now suffices to existentially check that the dimension of the quotient is at most $N$.
For this, observe that $\dim_K K[X_1, \dotsc, X_m]/I$ is equal to the $\overline K$-dimension of  $\overline K \otimes_K K[X_1, \dotsc, X_m]/I \cong \overline{K}[X_1, \dotsc, X_m]/(f_1, \dotsc, f_r)$.
Over $\overline{K}$ there is a quantifier-free formula expressing that the dimension is at most $N$, because the theory of algebraically closed fields has quantifier elimination, and this formula also works over $K$. 

The last implication will be proven by contradiction. Suppose that there is some pair of elementary extensions such that $K^{**}/K^*$ is not regular. Then $K^{**} \otimes_{K^*}\overline{K^*}$ is not an integral domain, so there are nonzero $x,y \in K^{**}\otimes_{K^*}\overline{K^*}$ such that $xy = 0$. There are $x_1, \dots, x_p, y_1, \dots, y_q \in \overline{K^*}$ and $l_1, \dots, l_p, r_1, \dots, r_q \in K^{**}$ such that 
\begin{equation*}
    x = \sum_{i=1}^p l_i \otimes x_i \text{ and } y = \sum_{j=1}^q r_j \otimes y_j.
\end{equation*}
Then $L = K^*(x_1,\dots, x_p, y_1, \dots, y_q)$ is a finite field extension of $K^*$. Let $K^*[X_1, \dots, X_p, Y_1, \dots, Y_q]$ be the polynomial ring in $p+q$ variables and let $\pi: K^*[X_1, \dots, X_p, Y_1, \dots, Y_q] \to L$ be the canonical homomorphism. If we define $\mathfrak{m} = \ker(\pi)$, then $L \cong K^*[X_1, \dots, X_p, Y_1, \dots, Y_q]/\mathfrak{m}$ and $\mathfrak{m}$ is a maximal ideal. However, $\mathfrak{m}$ does not remain maximal when passing to $K^{**}$ since $K^{**} \otimes_{K^*} L$ is not an integral domain. The \L o\'s--Tarski theorem then gives that for the 3-tuples $(m,d,r)$ for which there exist generators for $\mathfrak{m}$ there is no existential formula expressing that the ideal is maximal.
\end{proof}

We are left with the connection between conditions (1) and (4) of the main theorem.
We first state a lemma on elementary extensions of fields with finite degree of imperfection.
Recall here that the degree of imperfection of the field $K$ is $[K : K^p]$ if $K$ has characteristic $p>0$; if $K$ has characteristic zero, we simply define the degree of imperfection to be $1$.


\begin{lemma}\label{lem:automatic_separability}
  Let $K$ be of finite degree of imperfection. Then for any pair of elementary extensions $K^{**}$ and $K^*$ of $K$ with $K^* \subseteq K^{**}$, the extension $K^{**}/K^*$ is separable.
\end{lemma}
In the next section we will see that the condition that $K$ has finite degree of imperfection is necessary.
\begin{proof}
  There is nothing to show in characteristic zero, so let $p = \operatorname{char}(K) > 0$. The extension $K^{1/p}/K$ is finite by assumption, hence elementary transfer implies $(K^{*})^{1/p} = K^* K^{1/p}$ and $(K^{**})^{1/p} = K^{**} K^{1/p}$.
  Therefore $K^{**} \otimes_{K^*} (K^*)^{1/p} = K^{**} \otimes_K K^{1/p} = (K^{**})^{1/p}$ is a field, hence $K^{**}$ is linearly disjoint from $(K^*)^{1/p}$ over $K^*$, which proves that $K^{**}/K^*$ is separable by \cite[Chapter VIII, Proposition 4.1]{LangAlgebra}.
\end{proof}

The following proposition now finishes the proof of the main theorem.
\begin{proposition}
  The field $K^*$ is relatively algebraically closed in $K^{**}$ for all elementary extensions $K^{**}\supseteq K^*$ of $K$ if and only if $K$ has the property that every formula of the form $\forall x \psi(x, \overline y)$ with $\psi$ quantifier-free is equivalent over $K$ to an existential formula with parameters.
\end{proposition}
\begin{proof}
  Assume first that $K$ satisfies the condition on formulae. For every $n > 1$ consider the formula $\varphi_n(y_0, \dotsc, y_{n-1}) = \forall x (x^n + y_{n-1}x^{n-1} + \dotsb + y_0 \neq 0)$. By assumption, this formula is equivalent over $K^*$ and $K^{**}$ to an existential formula with parameters in $K$.
  In particular, for every tuple $\overline{b}$ in $K^*$ with $K^* \models \varphi_n(\overline b)$ we also have $K^{**} \models \varphi_n(\overline b)$. This proves that every polynomial with coefficients in $K^*$ which has a zero in $K^{**}$ already has a zero in $K^*$, i.e.\ that $K^*$ is relatively algebraically closed in $K^{**}$, as desired.

  Assume conversely that $K$ satisfies the condition of relatively algebraically closed extensions, and let $\varphi(\overline y)$ be a formula of the given form. We may restrict to formulae of the form \[\forall x (f(x, \overline{y}) = 0 \vee g_1(x, \overline{y}) \neq 0 \vee \dotsb \vee g_n(x, \overline{y}) \neq 0) \] where $f$ and the $g_i$ are polynomials with coefficients in $K$, by using conjunctive normal form and replacing disjunctions $t = 0 \vee s = 0$ by $ts = 0$.
  By the \L{}oś--Tarski Theorem, it suffices to show that for two elementary extensions $K^{**}/K^*$ of $K$ and a tuple $\overline b$ in $K^*$, it always holds that $K^* \models \varphi(\overline b)$ implies $K^{**} \models \varphi(\overline b)$.
  Assume first that at least one of the polynomials $g_i(X, \overline{b}) \in K^*[X]$ is non-zero.
  By assumption, $K^*$ is relatively algebraically closed in $K^{**}$, i.e.\ the finitely many zeroes of the non-zero polynomial $g_i(X, \overline{b})$ in $K^{**}$ are already in $K^{*}$.
  Hence $K^{*} \models \varphi(\overline{b})$ implies $K^{**} \models \varphi(\overline{b})$ in this case.

  Otherwise all $g_i(X, \overline{b})$ are the zero polynomial, and the only way for $K^* \models \varphi(\overline{b})$ is for $f(X, \overline{b})$ to also be the zero polynomial (unless $K^*$ is finite, but then there is nothing to prove). Hence $K^{**} \models \varphi(\overline{b})$ is trivially true as well.
\end{proof}

\begin{remark}
  The statements of the main theorem in some ways complement results obtained in \cite[Remark 5.3.6]{PhDThesisDittmann}. There it is observed that $K$ having an existential criterion for a polynomial to be irreducible, a special case of point (2) of the main theorem, is equivalent to the property that all pairs of elementary extensions $K^*\subseteq K^{**}$ have that $K^*$ is relatively algebraically closed in $K^{**}$ (see also \cite[Corollary 5.4]{Dittmann18}). There are also some equivalent conditions of a geometric kind, see \cite[Theorem 5.4.1]{PhDThesisDittmann}.
\end{remark}

\section{Examples, counterexamples and an open problem}

The most basic examples of fields satysfying the conditions from the main theorem are those with a model complete diagram. Any elementary extension of fields is regular (since irreducible polynomials over the ground field remain irreducible in the extension, and $p$-independent elements remain $p$-independent), so these fields satisfy the first condition from the main theorem.

This includes in particular the algebraically closed fields, real closed fields and $p$-adically closed fields, as these are model complete even in the original language of rings.
Likewise, this also includes every separably closed field of finite imperfection degree (see \cite[Lemma 3.2]{messmerSeparablyClosedFields}), as well as certain henselian fields such as $\mathbb{C}((T))$ (since here the valuation ring is both existentially and universally definable with parameters, both value group and residue field have model complete diagram, and an Ax--Kochen/Ershov principle holds).

A wider class of easy examples satisfying the conditions of the main theorem -- which in fact includes all the examples mentioned above -- is given by the fields with finite degree of imperfection and only finitely many Galois extensions of any given degree, equivalently the fields $K$ such that for every $n$ all extensions of $K$ of degree $n$ (in a fixed algebraic closure) are contained in some finite extension of $K$. (One may want to call such fields \emph{bounded}; unfortunately, there are already several definitions of such a notion in the literature, with different meanings for imperfect fields.)

The proof of this statement is similar to that of Lemma \ref{lem:automatic_separability}.
\begin{proposition}
  Let $K$ be a field such that for every $n$ all extensions of $K$ of degree $n$ in a fixed algebraic closure are contained in some finite extension of $K$.
  Then every pair of elementary extensions with $K^{**}\supseteq K^*$ is regular.
\end{proposition}
\begin{proof}
  For given $n>0$, let $K_n/K$ be a finite extension into which any extension of $K$ of degree $n$ embeds.
  Choosing a $K$-basis of $K_n$, the field $K_n$ is isomorphic to a $K$-vector space $K^m$ with a bilinear multiplication function $K^m \times K^m \to K^m$.
  In this way, the fact that every field extension of $K$ of degree $n$ embeds into the extension $K_n$ is expressible by a first-order sentence with parameters in $K$. By first-order transfer, any extension of $K^\ast$ of degree $n$ then embeds into the extension field $K^\ast K_n$.
  This implies that $\overline{K^*} = K^* \otimes_K \overline K$. Therefore regularity of $K^{**}/K$ implies that $K^{**} \otimes_{K^*} \overline{K^*} \cong K^{**} \otimes_K \overline{K}$ is a field, so $K^{**}/K^*$ is regular.
\end{proof}

More complicated examples for fields satisfying the conditions of the main theorem -- and in fact the motivating examples for this article -- are given by global fields, for which the first condition follows from \cite[Corollary 5.3]{Dittmann18} (relative algebraic closedness of $K^*$ in $K^{**}$) together with Lemma \ref{lem:automatic_separability}.
The main theorem now yields the following.
\begin{corollary}\label{cor:global_fields}
  Over global fields, primality and maximality of polynomial ideals can be expressed by existential formulae, and formulae with one existential quantifier are equivalent to universal formulae (all with parameters).
\end{corollary}
Note that global fields are far from being model complete, as following Rumely \cite{RumelyUndecidabilityGlobalFields} the ring $\mathbb{Z}$ is interpretable in every global field (in fact every global field is bi-interpretable with $\mathbb{Z}$), and hence every global field has many definable subsets (for instance the complement of the halting set in a suitably nice copy of $\mathbb{Z}$) which are not computably enumerable and therefore not existentially definable.

For a counterexample to the conditions of the main theorem we consider the following (see \cite[Example 5.3.7]{PhDThesisDittmann}).
Let $K = K_0(t_1, t_2, \dotsc)$ be a field obtained by adjoining countably many transcendental elements to an arbitrary base field $K_0$, and fix $n > 1$. For every $m$ write $L_m = K(\sqrt[n]{t_m})$. Observe that $L_m$ is isomorphic to $K$ over $K_0(t_1, \dotsc, t_{m-1})$.
Let $\mathfrak{U}$ be a non-principal ultrafilter on $\mathbb{N}$ and let $K_\infty = \prod_m K / \mathfrak{U}$ and $L_\infty = \prod_m L_m / \mathfrak{U}$ be ultraproducts with respect to $\mathfrak{U}$. Then $K_\infty$ and $L_\infty$ are isomorphic over $K$,
hence in particular both $L_\infty$ and $K_\infty$ are elementary extensions of $K$, but $L_\infty/K_\infty$ is an extension of degree $n$ under the natural embedding and hence not regular. Thus the first condition of the main theorem is violated.


The results of this paper leave an open question. It is of course true that every field $K$ with the condition that every pair of elementary extensions $K^{**} \supseteq K^*$ is a regular extension also has the property that in every such pair of extensions $K^\ast$ is relatively algebraically closed in $K^{\ast\ast}$. In fields with finite degree of imperfection, the converse is also true (Lemma \ref{lem:automatic_separability}).
It would be interesting to know whether this can fail if the degree of imperfection is infinite.

There are fields where not every pair of elementary extensions is separable. For example, take $K = \mathbb{F}_p(t_1, t_2, \dots)^{\text{sep}}$ (the separable closure of the field $\mathbb{F}_p(t_1,t_2,\dots)$) with the elementary extensions $K^* = \mathbb{F}_p(a,b,t_1,t_2, \dots)^{\text{sep}}$ and $K^{**} = \mathbb{F}_p(a,b,x,y,t_1,t_2, \dots)^{\text{sep}}$, where $a$, $b$ and $x$ are transcendental and $y$ satisfies $ax^p+by^p=1$ (regular extensions of separably closed fields of the same degree of imperfection are elementary by \cite[Theorem 1]{Wood_SepClosedFields}). Then $K^{**}/K^*$ is not separable, but $K^*$ is relatively algebraically closed in $K^{**}$ (see \cite[Remark 2.7.6]{Fried08}). The field $K$ is however not a counterexample to the above question, because it also has pairs of elementary extensions that are not relatively algebraically closed, for example $K^* = \mathbb{F}_p(a, t_1,t_2, \dots)^{\text{sep}}$ and $K^{**} = \mathbb{F}_p(\sqrt[p]{a}, t_1, t_2, \dots)^{\text{sep}}$.
This is despite separably closed fields being model complete in a very natural expansion of the language of rings (see again \cite[Theorem 1]{Wood_SepClosedFields}).

\bibliographystyle{amsalpha}
\bibliography{References}

\providecommand{\bysame}{\leavevmode\hbox to3em{\hrulefill}\thinspace}
\providecommand{\MR}{\relax\ifhmode\unskip\space\fi MR }
\providecommand{\MRhref}[2]{%
  \href{http://www.ams.org/mathscinet-getitem?mr=#1}{#2}
}
\providecommand{\href}[2]{#2}
\begin{thebibliography}{vdDS84}

\bibitem[Bou07]{BourbakiAlg5}
Nicolas Bourbaki, \emph{Algèbre, chapitres 4 à 7}, Springer-Verlag,
  Berlin-Heidelberg, 2007.

\bibitem[Dit18a]{Dittmann18}
Philip Dittmann, \emph{Irreducibility of polynomials over global fields is
  diophantine}, Compos. Math. \textbf{154} (2018), no.~4, 761--772.

\bibitem[Dit18b]{PhDThesisDittmann}
\bysame, \emph{A model-theoretic approach to the arithmetic of global fields},
  doctoral thesis, University of Oxford. Available at
  \url{https://ora.ox.ac.uk/objects/uuid:c798e052-b305-4cca-8964-e7959e486a5d},
  2018.

\bibitem[FJ08]{Fried08}
Michael~D. Fried and Moshe Jarden, \emph{Field arithmetic}, third ed.,
  Ergebnisse der Mathematik und ihrer Grenzgebiete. 3. Folge, vol.~11,
  Springer-Verlag, Berlin, 2008, revised by Jarden.

\bibitem[Hod97]{Hodges97}
Wilfrid Hodges, \emph{A shorter model theory}, Cambridge University Press,
  Cambridge, 1997.

\bibitem[Lan93]{LangAlgebra}
Serge Lang, \emph{Algebra}, third ed., Addison-Wesley, 1993.

\bibitem[Mes96]{messmerSeparablyClosedFields}
Margit Messmer, \emph{Some model theory of separably closed fields}, Model
  theory of fields, Lecture Notes in Logic, vol.~5, Springer-Verlag, Berlin,
  1996.

\bibitem[Rum80]{RumelyUndecidabilityGlobalFields}
Robert~S.\ Rumely, \emph{Undecidability and definability for the theory of
  global fields}, Trans. Amer. Math. Soc. \textbf{262} (1980), no.~1, 195--217.

\bibitem[vdD79]{Dries78}
Lou van~den Dries, \emph{Algorithms and bounds for polynomial rings}, Logic
  {C}olloquium '78 ({M}ons, 1978), Stud. Logic Foundations Math., vol.~97,
  North-Holland, Amsterdam-New York, 1979, pp.~147--157.

\bibitem[vdDS84]{Dries84}
Lou van~den Dries and Karsten Schmidt, \emph{Bounds in the theory of polynomial
  rings over fields. {A} nonstandard approach}, Invent. Math. \textbf{76}
  (1984), no.~1, 77--91.

\bibitem[Woo79]{Wood_SepClosedFields}
Carol Wood, \emph{Notes on the stability of separably closed fields}, J.
  Symbolic Logic \textbf{44} (1979), no.~3, 412--416.

\end{thebibliography}

\end{document}